\documentclass[12pt,twoside]{amsart}
\usepackage{amssymb,amsmath,amsthm, amscd, enumerate, mathrsfs, upgreek}
\usepackage{graphicx, hhline, tikz}
\usepackage[colorlinks=true,pagebackref,hyperindex]{hyperref}
\usepackage[all]{xy}
\usepackage{textcomp}
\usepackage{color}
\usepackage[backrefs, alphabetic, initials]{amsrefs}
\usepackage{color} 
\usepackage{latexsym}
\usepackage[T1]{fontenc} 
\usepackage{fancyhdr}
\usepackage{enumerate}
\usepackage{tikz-cd}
\usepackage{array, longtable}
\usepackage{caption}
\usepackage{mathtools}
\newcolumntype{C}{>{$}c<{$}}
\newcolumntype{L}{>{$}l<{$}}

\title[Kawamata--Miyaoka type inequality]
{A Kawamata--Miyaoka type inequality for Fano varieties of arbitrary Picard number}

\date{\today, version 0.01}

\author{Haidong Liu}
\address{Sun Yat-sen University, School of Mathematics, Guangzhou, 510275, China}
\email{liuhd35@mail.sysu.edu.cn, jiuguiaqi@gmail.com}

\DeclareMathOperator{\rank}{rank}

\DeclareMathOperator{\mult}{mult}

\DeclareMathOperator{\reg}{reg}
\DeclareMathOperator{\Sing}{Sing}

\DeclareMathOperator{\Cl}{Cl}

\newcommand\lcm{{\text{l.c.m.}}}

\newtheorem{thm}{Theorem}[section]
\newtheorem{lem}[thm]{Lemma}
\newtheorem{prop}[thm]{Proposition}

\newtheorem{cor}[thm]{Corollary}

\theoremstyle{definition}
\newtheorem{ex}[thm]{Example}
\newtheorem{defn}[thm]{Definition}
\newtheorem{rem}[thm]{Remark}

\makeatletter
 
 \@addtoreset{equation}{section}
\makeatother

\begin{document}

\begin{abstract}
    Let $X$ be a $\mathbb Q$-factorial canonical weak Fano variety of dimension $n\geq 2$. We show that if the $\mathbb Q$-Fano index $q_{\mathbb Q}(X)\geq 3$, then $X$ satisfies a Kawamata--Miyaoka type inequality:
    \[
        c_1(X)^n\leq 4\,\hat c_2(X)\cdot c_1(X)^{n-2}.
    \]
    As an application, we show that the $\mathbb Q$-Fano index of a Gorenstein canonical Fano $3$-fold lies in the set $\{m\in\mathbb Z_{>0}\mid m\leq 22\} \cup\{24,30,42\}$.
\end{abstract}

\maketitle 
\tableofcontents

\section{Introduction}\label{sec1}

A normal projective variety is said to be \emph{Fano} (resp., \emph{weak Fano}) if its anti-canonical divisor is ample (resp., nef and big). In the minimal model program,  Fano varieties with mild singularities (e.g., terminal or canonical singularities) are among the fundamental building blocks of algebraic varieties.

Numerical invariants such as the Picard number, the anti-canonical degree, and the Fano index (see \S~\ref{subsec.fi} for precise definitions) play an important role in the classification of terminal or canonical Fano varieties. 
A natural first step in such a classification is to determine the explicit bounds for these invariants. A key technique for obtaining such bounds is the so-called \emph{Kawamata--Miyaoka type inequality} for $\mathbb Q$-factorial weak Fano varieties, whenever the second Chern class $c_2(X)$ is well defined:
\begin{equation}\label{eq.KM}
    c_1(X)^n\leq b_n \, c_2(X)\cdot c_1(X)^{n-2}.
\end{equation}
This was proved in \cite{ijl}*{Corollary 1.7} for the terminal case. Later in \cites{liu-liu2023, liu-liu2024}, the bound $b_n$ was improved to $4$ (and to $3$ for terminal Fano $3$-folds) under the Picard number one condition, by studying rank one subsheaves of the tangent bundle $\mathcal T_X$. Using the same approach and the generalized Chern classes in \cite{jiang-liu-liu}, the Kawamata--Miyaoka type inequality was further extended to Fano varieties of Picard number one with $\epsilon$-lc singularities. These inequalities turned out to be very effective in establishing upper bounds for degrees \cite{jiang-liu-liu}*{Theorem 1.1} and for Fano indices \cite{jiang-liu2025}*{Theorem 1.1}.

In \S~\ref{sec.KM} of this paper, we establish the bound $b_n=4$ in \eqref{eq.KM} for all canonical weak Fano varieties with $\mathbb Q$-Fano index $q_{\mathbb Q}(X)\geq 3$, without the Picard number one restriction. In fact, we prove that $4\,\hat c_2(X)-c_1(X)^2$ is pseudoeffective in this setting (see \S~\ref{sub.langerineq} for the definition of the generalized second Chern class $\hat c_2$). 

\begin{thm}\label{thm.KMinq}
    Let $X$ be a $\mathbb Q$-factorial canonical weak Fano variety of dimension $n\geq 2$. Suppose the $\mathbb Q$-Fano index $q_{\mathbb Q}(X)\geq 3$. Then for any nef divisors $D_1,\dots ,D_{n-2}$, 
    \[
        \left( 4\,\hat c_2(X)-c_1(X)^2\right)\cdot D_1\cdots D_{n-2}\geq 0.
    \]
    Moreover, if $X$ is Fano and the divisors $D_i$ are ample, then
    \[
        \left( 4\,\hat c_2(X)-c_1(X)^2\right)\cdot D_1 \cdots D_{n-2}> 0.
    \]
\end{thm}

Applying Theorem \ref{thm.KMinq} to $D_i=c_1(X)=-K_X$ for any $i$, we immediately have a Kawamata--Miyaoka type inequality as follows.

\begin{cor}\label{cor.KMinq}
    Let $X$ be a $\mathbb Q$-factorial canonical weak Fano variety of dimension $n\geq 2$. Suppose the $\mathbb Q$-Fano index $q_{\mathbb Q}(X)\geq 3$. Then  
    \[
        c_1(X)^n\leq 4\,\hat c_2(X)\cdot c_1(X)^{n-2}.
    \]
    Moreover, the inequality is strict if $X$ is Fano.
\end{cor}

As the Picard number one case, Corollary \ref{cor.KMinq} is expected to be useful in classifying terminal or canonical Fano varieties of any Picard number. We present one such application in Section \ref{sec.Gorenstein}, with an introduction in the following subsection.

\subsection{Application on Fano indices of Gorenstein canonical Fano $3$-folds}

An interesting phenomenon in the minimal model program is that the index of a log canonical singularity $P\in X$, the index of a terminal Calabi--Yau variety, and the Fano index of a canonical weighted projective space $X$  $\cup ~\{m\in\mathbb Z_{>0}\mid m\leq \dim X\}$ all coincide in low dimensions. Specifically, the common index set is 
\[
    I_2\coloneq \{m\in\mathbb Z_{>0}\mid \varphi(m)\leq 2\}=\{1,2,3,4,6\}
\]
for surfaces (see \cite{fujino}, \cite{machida-oguiso}, and the fact that $q_{\mathbb Q}(S)=3, 4,6$ for $S=\mathbb P^2, \mathbb P^1(1,1,2), \mathbb P^1(1,2,3)$, respectively), and 
\[
    I_3\coloneq \{m\in\mathbb Z_{>0}\mid \varphi(m)\leq 20\} \setminus \{60\}
\]
for $3$-folds (see \cite{fujino}*{Theorem 0.1},  \cite{masamura}*{Proposition 3.2}, and \cite{kasprzyk}*{Table 3}), where $\varphi(m)$ is the Euler function of $m$. In particular, the number in $I_3$ is at most $66$.

In \cite{jiang-liu2025}, Jiang and the author proved that the $\mathbb Q$-Fano index of a canonical Fano $3$-fold is also bounded above by $66$. A natural next step in classifications is to determine the complete set $I_{3, can}$ of all possible $\mathbb Q$-Fano indices of canonical Fano $3$-folds, where 
\[
    I_3\subseteq I_{3, can}\subseteq \{m\in\mathbb Z_{>0}\mid m\leq 66\}.
\]

In general, one cannot expect $I_3=I_{3,can}$. This is suggested by the existence of a canonical Fano surface $S$ with $q_{\mathbb Q}(S)=5$ (see, e.g., \cite{ye}*{theorem 1.2} and \cite{miyanishi-zhang}*{Lemma 3}), which implies $I_2\subsetneq I_{2, can}$. In \S~\ref{sec.Gorenstein}, we use Corollary \ref{cor.KMinq} to show that if $q_{\mathbb Q}(X)\notin I_3$, then $X$ is not Gorenstein.

\begin{thm}[Corollary \ref{cor.index}]
    For any Gorenstein canonical Fano $3$-fold $X$, its $\mathbb Q$-Fano index $q_{\mathbb Q}(X)$ lies in the set $\{m\in\mathbb Z_{>0}\mid m\leq 22\} \cup\{24,30,42\}\subsetneq  I_3$.
\end{thm}

\section{Preliminaries}

Throughout this paper, we work over the complex number field $\mathbb C$. We will freely use the basic notation in \cites{kollar-mori,lazarsfeld}.

\subsection{Singularities}

Let $X$ be a normal variety such that $K_X$ is $\mathbb Q$-Cartier. The \emph{Gorenstein index} $r_X$ of $X$ is defined to be the smallest positive integer $m$ such that $mK_X$ is Cartier. We say $X$ is \emph{Gorenstein} if $r_X=1$. Let $f\colon Y\to X$ be a resolution of singularities and write
\[
    K_Y=f^*K_X+\sum_Ea(E,X)E,
\]
where $E$ runs over all prime divisors on $Y$ and $a(E,X)\in \mathbb Q$. We say that $X$ has \emph{terminal} (resp., \emph{canonical}, \emph{klt}, \emph{$\epsilon$-lc} for some fixed $0\leq \epsilon\leq 1$) singularities if $a(E,X)>0$ (resp., $a(E,X)\geq 0$, $a(E,X)>-1$, $a(E,X)\geq \epsilon-1$) for any exceptional divisor $E$ over $X$. For simplicity, we say that $X$ is \emph{terminal}, \emph{canonical}, \emph{klt}, or \emph{$\epsilon$-lc}, respectively.

\subsection{Degrees and Fano indices}\label{subsec.fi}

Let $X$ be a canonical weak Fano variety of dimension $n$. The \emph{(anti-canonical) degree} of $X$ is defined to be $(-K_X)^n=c_1(X)^n$ and the \emph{$\mathbb Q$-Fano index} of $X$ is defined to be
\begin{equation*}
    q_{\mathbb Q}(X) \coloneq \max\{q \mid -K_X \sim_{\mathbb Q}qA, \, A\in \Cl (X) \}. 
\end{equation*}
It is known that $\Cl (X)$ is a finitely generated Abelian group, so $q_{\mathbb Q}(X)$ is a positive integer. For more details, see \cite{ip}*{\S\,2} or \cite{prokhorov2010}.

Recall some basic facts about degrees and  their relations to Fano indices and other invariants of Gorenstein canonical Fano $3$-folds.

\begin{prop}\label{thm.degree}
    Let $X$ be a Gorenstein canonical Fano $3$-fold. Then 
    \begin{enumerate}
        \item $c_1(X)^3$ is an even number;
        \item $c_1(X)^3\leq 72$, and the equality holds if and only if $X\cong \mathbb P(1,1,1,3)$ or $\mathbb P(1,1,4,6)$, where $q_{\mathbb Q}(X)=6$ or $12$ respectively. 
    \end{enumerate}
\end{prop}

\begin{proof}
    The first conclusion follows directly from the Riemann--Roch formula $\chi(-K_X)=\frac{c_1(X)^3}{2}+3$. The second  conclusion is proved by Prokhorov \cite{prokhorov2005}.
\end{proof}

\begin{thm}[\cite{jiang-liu-liu}*{Theorem~4.2}]\label{thm.degreeandindex}
    Let $X$ be a Gorenstein canonical Fano $3$-fold. Let $A$ be an ample Weil divisor on $X$ such that $-K_X\sim qA$ for some positive integer $q$. Let $J_A$ be the smallest positive integer such that $J_AA$ is Cartier in codimension $2$. Then
    \[
        J_A\mid q \text{ and } q^2\mid J_Ac_1(X)^3.
    \] 
\end{thm}

\subsection{$\mathbb Q$-variant of Langer's inequality}\label{sub.langerineq}

Let $X$ be a normal projective variety with canonical singularities of dimension $n\geq 2$. Then in codimension 2, $X$ has only quotient singularities and admits a structure of a $\mathbb{Q}$-variety (see \cite{gkpt}*{\S\,3.2 and \S\,3.6}). Then as in \cite{gkpt}*{Construction~3.8 and \S\,3.7}, for any reflexive sheaf $\mathcal{E}$ on $X$, we can define the \emph{generalized second Chern class} $\hat{c}_2(\mathcal{E})$, which is a symmetric $\mathbb{Q}$-multilinear map:
\[
 \hat{c}_2(\mathcal{E})\colon N^1(X)_{\mathbb{Q}}^{\times (n-2)} \longrightarrow \mathbb{Q},\quad (\alpha_1,\dots,\alpha_{n-2})\longmapsto \hat{c}_2(\mathcal{E})\cdot \alpha_1\cdots\alpha_{n-2}.
\]
On the other hand, as $X$ is $\mathbb{Q}$-factorial in codimension $2$, for any two $\mathbb{Q}$-divisor classes $\beta$ and $\gamma$ on $X$, the product $\beta\cdot \gamma$ is also a well-defined symmetric $\mathbb{Q}$-multilinear map:
\[
 \beta\cdot \gamma\colon N^1(X)_{\mathbb{Q}}^{\times (n-2)}\longrightarrow \mathbb{Q},\quad (\alpha_1,\dots,\alpha_{n-2})\longmapsto \beta\cdot \gamma\cdot \alpha_1\cdots\alpha_{n-2}.
\]
Now the \emph{$\mathbb{Q}$-Bogomolov discriminant} $\hat{\Delta}(\mathcal{E})$ of a reflexive sheaf $\mathcal{E}$ of rank $r$ is defined to be
	\[
	\hat{\Delta}(\mathcal{E})\coloneqq 2r \hat{c}_2(\mathcal{E}) - (r-1) c_1(\mathcal{E})^2,
	\]
which is viewed as a symmetric $\mathbb{Q}$-multilinear map $N^1(X)_{\mathbb{Q}}^{\times (n-2)}\rightarrow \mathbb{Q}$. 
Let $D_1,\dots, D_{n-1}$ be a collection of nef $\mathbb Q$-Cartier divisors. 
According to \cite{keel-matsuki-mckernan}*{Lemma~6.5}, for any semistable reflexive sheaf $\mathcal{E}$ with respect to $(D_1,\dots, D_{n-1})$, the following \emph{$\mathbb{Q}$-Bogomolov--Gieseker inequality} holds:
\begin{equation}\label{eq.BGineq}
 \hat{\Delta}(\mathcal{E})\cdot D_1\cdots D_{n-2} \geq 0.
\end{equation}

The following inequality generalizes the $\mathbb{Q}$-Bogomolov--Gieseker inequality to the non-semistable cases, and is commonly referred to as the $\mathbb Q$-variant of Langer's inequality (see, for example, \cite{jiang-liu-liu}*{\S~3.1} for details).

\begin{thm}\label{thm.LangerIneqII}
    Let $X$ be a normal projective variety with canonical singularities of dimension $n\geq 2$. Let $\mathcal{E}$ be a reflexive sheaf of rank $r$ on $X$. Let $(D_1,\dots,D_{n-1})$ be a collection of nef $\mathbb Q$-Cartier divisors. Then we have
	\[
	(D_1\cdots D_{n-2} \cdot D_{n-1}^2) (\hat{\Delta}(\mathcal{E})\cdot D_1\cdots D_{n-2}) + r^2 (\mu_{\max} - \mu) (\mu-\mu_{\min}) \geq 0,
	\]
	where $\mu$ (resp., $\mu_{\max}$ and $\mu_{\min}$) is the slope (resp., maximal slope and minimal slope) of $\mathcal{E}$ with respect to $(D_1,\dots,D_{n-1})$.
\end{thm}

\subsection{Basic notions of foliations}
	
    In this subsection, we collect some basic notions and facts regarding foliations on normal varieties. A more detailed explanation can be found in \cite{aroujo-druel}, \cite{druel2021}*{\S\,3} and the references therein. 
    
    \begin{defn}
        A {\it foliation} on a normal variety $X$ is a non-zero coherent subsheaf $\mathcal{F}$ of the tangent sheaf $\mathcal{T}_X$ such that 
        \begin{enumerate}
    	\item $\mathcal{T}_X/\mathcal{F}$ is torsion-free, and	
    	\item $\mathcal{F}$ is closed under the Lie bracket.
        \end{enumerate} 
        The {\it canonical divisor} of a foliation $\mathcal{F}$ is any Weil divisor $K_{\mathcal{F}}$ on $X$ such that $\det(\mathcal{F})\cong \mathcal{O}_X(-K_{\mathcal{F}})$. 
     
        Let $X_{0}\subset X_{\reg}$ be the largest open subset over which $\mathcal{T}_X/\mathcal{F}$ is locally free. A \emph{leaf} of $\mathcal{F}$ is a maximal connected and immersed holomorphic submanifold $L\subset X_{0}$ such that $\mathcal{T}_L=\mathcal{F}|_L$. A leaf is called \emph{algebraic} if it is open in its Zariski closure, and a foliation $\mathcal{F}$ is said to be \emph{algebraically integrable} if its leaves are algebraic. 
    \end{defn}

    Let $\mathcal{F}$ be an algebraically integrable foliation on a normal projective variety $X$. Then there exists a diagram as follows (\cite{druel2021}*{\S\,3.6}), which is called the \emph{family of leaves}: 

    \begin{equation}\label{eq.familyofleaves}
        \begin{tikzcd}[row sep=large, column sep=large]
    	U \arrow[r,"e"] \arrow[d,"f"]
    	& X \\
    	T
        \end{tikzcd}
    \end{equation}
    where $U$ and $T$ are normal projective varieties, $e$ is birational and $f$ is an equidimensional fibration, and the image of a general fiber $F$ of $f$ is the closure of a leaf of $\mathcal{F}$. In particular, $\dim F=\rank \mathcal F$. If $F$ is rationally connected, then $\mathcal F$ is said to be \emph{a foliation with rationally connected general leaves}.

    The following proposition restates \cite{ou}*{Proposition 3.1}, which is a special case of \cite{druel2017}*{Proposition 4.1}. It will serve a similar role in the proof of Proposition \ref{prop.diff} to that in establishing the pseudoeffectivity of the second Chern class in \cite{ou}.

    \begin{prop}\label{prop.pseff}
        Let $f\colon U\to T$ be a surjective morphism between normal projective varieties, where $U$ is $\mathbb Q$-factorial. Let $F$ be a general fiber of $f$. Let $\Delta$ be an effective $\mathbb Q$-divisor on $U$ such that 
        $(F,\Delta|_F)$ is log canonical. If $\kappa(F, K_F+\Delta|_F)\geq 0$, then $K_{\mathcal F}+\Delta$ is pseudoeffective, where $\mathcal F$ is the foliation induced by $f$. 
    \end{prop}

\section{Kawamata--Miyaoka type inequality}\label{sec.KM}

As shown in \cites{ijl,liu-liu2023, jiang-liu-liu}, the key step in establishing the Kawamata--Miyaoka type inequality is to study the foliation induced by the maximal destabilizing subsheaf $\mathcal E$ of the tangent sheaf $\mathcal T_X$. For $\rank (\mathcal E)\geq 2$, the desired inequality follows directly from Langer's inequality, regardless of the Picard number; for $\rank (\mathcal E)=1$, the key result is \cite{liu-liu2023}*{Proposition 3.8} or \cite{jiang-liu-liu}*{Proposition 3.6}, where the assumption $\rho(X)=1$ is essential. In this section, we generalize \cite{liu-liu2023}*{Proposition 3.8} and \cite{jiang-liu-liu}*{Proposition 3.6} to weak Fano varieties of arbitrary Picard number. 

\begin{prop}\label{prop.diff}
    Let $X$ be a weak Fano variety of dimension $n\geq 2$. Let $A$ be a Weil divisor on $X$ such that $-K_X\sim_{\mathbb Q} qA$ for some rational number $q>2$. Let $\mathcal L$ be a rank one algebraically integrable foliation with rationally connected general leaves on $X$. 
    \begin{enumerate}
        \item If $X$ has canonical singularities, then $-K_X+2K_{\mathcal L}$ is pseudoeffective.
        \item If $X$ has terminal singularities, then $-r_XK_X+(2r_X+1)K_{\mathcal L}$ is pseudoeffective, where $r_X$ is the Gorenstein index.
        \item If $X$ has $\epsilon$-lc singularities, then $-K_X+(1+\epsilon)K_{\mathcal L}$ is pseudoeffective, where $0<\epsilon\leq 1$ is a real number.
    \end{enumerate}
\end{prop}

\begin{proof}
    The rank one algebraically integrable foliation  $\mathcal L$ induces a family of leaves 
    \begin{equation}\label{eq.folwithrc}
    \begin{tikzcd}[row sep=large, column sep=large]
	U \arrow[r,"e"] \arrow[d,"f"]
	& X \\
	T
    \end{tikzcd}
    \end{equation}
    where the general fiber $F$ of $f$ is isomorphic to $\mathbb P^1$. We assume that $X$ has only canonical singularities. Then there exists an effective $e$-exceptional $\mathbb Q$-divisor $\Delta$ on $U$ such that  
    \[
        K_U-e^*K_X\sim_{\mathbb Q} \Delta.
    \]
    Then it follows from $K_U\cdot F=\deg K_F=-2$ that
    \[
        -e^*K_X\cdot F=\Delta\cdot F+2\geq 2.
    \]
    
    We claim that there exists an $f$-horizontal $e$-exceptional prime divisor on $U$. Otherwise, the $e$-exceptional divisor $\Delta$ would be $f$-vertical, and hence  
    \[
        2=-K_U\cdot F=-e^*K_X\cdot F=-K_X\cdot e(F)=qA\cdot e(F)\geq q,
    \]
    as $A$  is Cartier in a neighborhood of $e(F)\cong \mathbb P^1$. This contracts our assumption.
    
    Let $E$ be an $f$-horizontal $e$-exceptional prime divisor on $U$. Then $E$ is Cartier in a neighborhood of $F$, so $0<a\coloneq 1/ E\cdot F\leq 1$. Now we have
    \begin{equation}\label{eq.keyeq}
        K_U+aE-\frac{1}{2}e^*K_X\sim_{\mathbb Q} \Delta+aE+\frac{1}{2}e^*K_X,
    \end{equation}
    where $(\Delta+aE+\frac{1}{2}e^*K_X)\cdot F= \Delta \cdot F +1 -\frac{1}{2}(\Delta\cdot F+2)=\frac{1}{2}\Delta\cdot F\geq 0$. That is, $ (\Delta+aE+\frac{1}{2}e^*K_X)|_F$ is effective on $F\cong \mathbb P^1$. 

    Let $H$ be an ample divisor on $U$ and $\varepsilon>0$ be a rational number. Then $-\frac{1}{2}e^*K_X+\varepsilon H$ is ample as $-e^*K_X$ is nef and big. Note that $0<a\leq 1$. So we can choose an effective $\mathbb Q$-divisor $B$ general enough  on $U$ such that $B\sim_{\mathbb Q} -\frac{1}{2}e^*K_X+\varepsilon H$ and $(F, aE|_F+B|_F)$ is log canonical. Moreover, we have
    \[
        K_F+aE|_F+B|_F \sim_{\mathbb Q}(K_U+aE-\frac{1}{2}e^*K_X+\varepsilon H)|_F\sim_{\mathbb Q} (\Delta+aE+\frac{1}{2}e^*K_X+\varepsilon H)|_F,
    \]
    which implies that $\kappa(F, K_F+aE|_F+B|_F)\geq 0$.
    
    Let $\pi \colon V\to U$ be a resolution isomorphic outside $\Sing(U)$ and $K_V\sim_{\mathbb Q}\pi^*K_U+\Delta_V$. Let $g\colon V\to T$ be the induced morphism such that $g=f\circ \pi$. As $U$ is smooth over the generic point of $T$, the $\pi$-exceptional divisor is $g$-vertical and $\pi$ is isomorphic over the generic point of $T$. So after replacing $U$ by $V$ if necessary without effecting the results on the general fiber $F$, we may assume that $U$ is smooth; in particular, $U$ is $\mathbb Q$-factorial.

    Then we can apply Proposition \ref{prop.pseff} to $f\colon U\to T$ and obtain that $K_{\mathcal F}+aE+B$ is pseudoeffective, where $\mathcal F$ is the foliation induced by $f$. As this holds for arbitrary $\varepsilon>0$, we obtain that $K_{\mathcal F}+aE-\frac{1}{2}e^*K_X$ is pseudoeffective by taking limit. Hence 
    \[
        K_{\mathcal L}-\frac{1}{2}K_X=e_*(K_{\mathcal F}+aE-\frac{1}{2}e^*K_X)
    \]
    is pseudoeffective.
    
     Finally, we deal with the terminal and $\epsilon$-lc cases. In the former case, we have that $\mult_E (r_X\Delta)\geq 1$, and hence $r_X\Delta\cdot F\geq E\cdot F\geq 1$. Note that
     \[
        (\Delta+aE+(\frac{\Delta\cdot F+1}{\Delta\cdot F+2})e^*K_X)\cdot F=\Delta\cdot F+1-\frac{\Delta\cdot F+1}{\Delta\cdot F+2}(\Delta\cdot F+2)=0.
     \]
     Then as above, we have that $K_{\mathcal L}-\frac{1}{\Delta\cdot F+2}K_X$ is pseudoeffective. Hence
     \[
        K_{\mathcal L}-\frac{r_X}{2r_X+1}K_X= K_{\mathcal L}-\frac{1}{\Delta\cdot F+2}K_X+\frac{r_X\Delta\cdot F-1}{(2r_X+1)(\Delta\cdot F+2)}(-K_X)
    \]
    is pseudoeffective.
    In the latter case, as $X$ has $\epsilon$-lc singularities, we have
    \[
        K_U-e^*K_X+\Delta_1\sim_{\mathbb Q} \Delta_2,
    \]
    where $\Delta_1$ and $\Delta_2$ are effective $e$-exceptional $\mathbb Q$-divisors without common components, and the coefficients of $\Delta_1$ are $\leq 1-\epsilon$. Assume that $\mult_E\Delta_1=1-\epsilon _1$. Then $ \epsilon\leq \epsilon_1\leq 1$ and $\Delta_1\cdot F\geq (1-\epsilon_1)E\cdot F\geq 1-\epsilon_1$. Hence we have
    \[
        K_U+\epsilon_1 aE+\Delta_1-\frac{1}{1+\epsilon_1 } e^*K_X\sim_{\mathbb Q} \Delta_2+\epsilon_1 aE+\frac{\epsilon_1 }{1+\epsilon_1 } e^*K_X,
    \]
    where $(\Delta_2+\epsilon_1 aE+\frac{\epsilon_1 }{1+\epsilon_1 } e^*K_X)\cdot F=\Delta_2\cdot F+\epsilon_1+\frac{\epsilon_1 }{1+\epsilon_1 }(\Delta_1\cdot F-\Delta_2\cdot F-2)=\frac{1}{1+\epsilon_1 }\Delta_2\cdot F+\frac{\epsilon_1}{1+\epsilon_1 }(\Delta_1\cdot F+\epsilon_1-1)\geq 0$. Then the same argument implies that 
    \[
        K_{\mathcal L}-\frac{1}{1+\epsilon} K_X=K_{\mathcal L}-\frac{1}{1+\epsilon_1 } K_X+\frac{\epsilon_1-\epsilon}{(1+\epsilon_1)(1+\epsilon)}(-K_X)
    \]
    is pseudoeffective. 
\end{proof}

\begin{rem}
    The assumption that $-K_X\sim_{\mathbb Q} qA$ for some $q>2$ is only used in the proof to ensure the existence of an $f$-horizontal $e$-exceptional prime divisor. Therefore, the proposition also holds for $q\leq 2$ provided that such a divisor exists. Moreover, the more $f$-horizontal $e$-exceptional divisors there are, the smaller $c_1(\mathcal L)\cdot c_1(X)^{n-1}/c_1(X)^n$ becomes. 
\end{rem}

\begin{rem}\label{rem.smooth}
    On the other hand, if there are no $f$-horizontal $e$-exceptional prime divisors, then the rank one foliation $\mathcal L$ gives rise to an almost holomorphic map $X\dasharrow T$ whose general fiber $F$ is $\mathbb P^1$. In particular, we have $\rho(X)\geq 2$ and $c_1(\mathcal L)\cdot F=qA\cdot F=-K_X\cdot F=2$. 
\end{rem}

The following observation for the case where $X$ is smooth and $q=2$ is due to Jie Liu. We reproduce his proof.

\begin{prop}[Jie Liu]\label{prop.smoothq=2}
    Let $X$ be a Fano manifold such that $-K_X\sim 2A$ for some ample divisor $A$. Let $\mathcal L$ be a rank one algebraically integrable foliation with rational connected leaves on $X$ such that $A\cdot F=1$ for a general leaf $F$ of $\mathcal L$. Then 
    \[
        2c_1(\mathcal L)\cdot c_1(X)^{n-1}\leq c_1(X)^n.
    \]
\end{prop}
\begin{proof}
    We use the same notation as in \eqref{eq.folwithrc}. Let $B\coloneq e^*A$, then $B$ is a Cartier divisor such that $B\cdot F=1$. Moreover, $B$ is $f$-ample as $A$ is ample and $e$ is finite on the fibers of $f$. It follows that any fiber of $f$ is reduced and irreducible. Therefore, $f$ is a $\mathbb P^1$-bundle by \cite{kollar}*{II, Theorem 2.8}. That is, $\mathcal E\coloneq f_*\mathcal O_U(B)$ is a rank two locally free sheaf on $T$ and $U\cong \mathbb P(\mathcal E)$. By taking a resolution of $T$ and the base change, we can further assume that $T$ is smooth. 

    Let $\mathcal F$ be the foliation induced by $f$. Then there exists an $e$-exceptional divisor $\Delta_f$ such that $c_1(\mathcal F)+\Delta_f=-K_{\mathcal F}+\Delta_f=-e^*K_{\mathcal L}=e^*c_1(\mathcal L)$ (see \cite{druel2021}*{\S~3.6} for example). By the projection formula, we have
    \[
        c_1(\mathcal F)\cdot B^{n-1}=(c_1(\mathcal F)+\Delta_f)\cdot B^{n-1}=c_1(\mathcal L)\cdot A^{n-1}.
    \]
    On the other hand, we have
    \[
        c_1(\mathcal F)=-K_{\mathcal F}=-K_{\mathbb P(\mathcal E)/T}=2B-f^*c_1(\mathcal E)=e^*c_1(X)-f^*c_1(\mathcal E),
    \]
    which implies that 
    \[
        c_1(\mathcal L)\cdot A^{n-1}= c_1(\mathcal F)\cdot B^{n-1}=(e^*c_1(X)-f^*c_1(\mathcal E))\cdot B^{n-1}=c_1(X)\cdot A^{n-1}-f^*c_1(\mathcal E)\cdot B^{n-1}.
    \]
    Note that $B^n-f^*c_1(\mathcal E)\cdot B^{n-1}+f^*c_2(\mathcal E)\cdot B^{n-2}=0$. As $B$ is nef, the locally free sheaf $\mathcal E$ is nef and big. Hence we have $f^*c_2(\mathcal E)\cdot B^{n-2}\geq 0$ (cf. \cite{lazarsfeld}*{Theorem 8.2.1}). It follows that 
    \[
        \frac{1}{2}c_1(X)\cdot A^{n-1}=A^n=B^n\leq f^*c_1(\mathcal E)\cdot B^{n-1}=c_1(X)\cdot A^{n-1}-c_1(\mathcal L)\cdot A^{n-1}.
    \]
    Then the desired inequality follows.
\end{proof}

We now turn to the proof of Theorem \ref{thm.KMinq}. The argument follows exactly as in \cite{jiang-liu-liu}*{Theorem 1.2} (or \cite{liu-liu2023}*{Theorem 1.1}) after replacing \cite{jiang-liu-liu}*{Proposition 3.6} (or \cite{liu-liu2023}*{Proposition 3.8}) with Proposition \ref{prop.diff}. We include the detailed proof below for the reader's convenience.

\begin{proof}[Proof of Theorem \ref{thm.KMinq}]
    If $\mathcal{T}_X$ is semistable with respect to the collection of nef divisors $(D_1,\dots,D_{n-2}, c_1(X))$, then by the $\mathbb{Q}$-Bogomolov--Gieseker inequality \eqref{eq.BGineq}, we have
    \[
        c_1(X)^{2}\cdot D_1 \cdots D_{n-2}\leq \frac{2n}{n-1}\, \hat{c}_2(X)\cdot D_1 \cdots D_{n-2}\leq 4 \, \hat{c}_2(X)\cdot D_1 \cdots D_{n-2}.
    \]
    If $n\geq 3$, then the second inequality above is strict; if $n=2$ and additionally $X$ is Fano, then by \cite{gkp}*{Theorem 1.2}, the first inequality  above is strict and hence we have
    \[
        c_1(X)^{2}\cdot D_1 \cdots D_{n-2} <4 \,\hat{c}_2(X)\cdot D_1 \cdots D_{n-2}.
    \]
    So we may assume that $\mathcal{T}_X$ is not semistable with respect to $(D_1,\dots,D_{n-2}, c_1(X))$. Applying Theorem~\ref{thm.LangerIneqII} to $\mathcal{T}_X$ and $(D_1,\dots,D_{n-2}, c_1(X))$ yields
    \begin{align*}
        \hat{\Delta}(\mathcal{T}_X)\cdot D_1 \cdots D_{n-2} + n\left(\mu_{\max}(\mathcal{T}_X)-\mu(\mathcal{T}_X)\right)
        & \geq \frac{n^2\left(\mu_{\max}(\mathcal{T}_X)-\mu(\mathcal{T}_X)\right) \cdot \mu_{\min}(\mathcal{T}_X)}{n\mu(\mathcal{T}_X)}.
    \end{align*}
    It follows from \cite{liu-liu2023}*{Proposition~3.6} that $\mu_{\min}(\mathcal{T}_X)\geq 0$, and it is strict if $X$ is Fano and the $D_i$'s are ample. As $\mu_{\max}(\mathcal{T}_X)>\mu(\mathcal{T}_X)$, the above inequalities imply that
    \begin{equation}\label{eq.keyineq}
        \left(2\hat{c}_2(X) - c_1(X)^2\right)\cdot D_1 \cdots D_{n-2} + \mu_{\max}(\mathcal{T}_X) \geq 0,
    \end{equation}
    and it is strict if $X$ is Fano and the $D_i$'s are ample. Denote by $\mathcal{F}$ the maximal destabilizing subsheaf of $\mathcal{T}_X$. Then \cite{liu-liu2023}*{Proposition~3.6} implies that 
    \begin{equation}\label{eq.genample}
        c_1(\mathcal{F})\cdot c_1(X)\cdot D_1 \cdots D_{n-2}\leq c_1(X)^2\cdot D_1 \cdots D_{n-2}.
    \end{equation}
    If $\rank (\mathcal F)\geq 2$, then
    \[
        \mu_{\max}(\mathcal{T}_X) = \mu(\mathcal{F})=\frac{c_1(\mathcal{F})\cdot c_1(X)\cdot D_1 \cdots D_{n-2}}{\rank (\mathcal{F})} \leq \frac{1}{2} \, c_1(X)^2\cdot D_1 \cdots D_{n-2}
    \]
    and the desired results follow immediately from \eqref{eq.keyineq}. If $\rank (\mathcal F)=1$, then as $\mathcal F$ is the maximal destabilizing subsheaf of $\mathcal{T}_X$, we have $\mu(\mathcal{F})>\mu(\mathcal T_X)>0$ and $\mu(\mathcal{F})>\mu(\mathcal T_X/\mathcal F)$. Then by \cite{ou}*{Proposition 2.2}, which extends \cite{cp}*{Theorem 1.1} to singular varieties, $\mathcal F$ is a rank one algebraically integrable foliation with rationally connected general leaves. Hence
    by Proposition \ref{prop.diff}, 
    \[
        \mu_{\max}(\mathcal{T}_X) = \mu(\mathcal{F})=c_1(\mathcal{F})\cdot c_1(X)\cdot D_1 \cdots D_{n-2} \leq \frac{1}{2} \, c_1(X)^2\cdot D_1 \cdots D_{n-2},
    \]
    so the desired results follow from \eqref{eq.keyineq} again.
\end{proof}

\begin{rem}\label{rem.lccase}
    A direct adaptation of the proof of Theorem \ref{thm.KMinq} via Proposition \ref{prop.diff}(3) yields a generalization to the $\epsilon$-lc setting. Specifically, let $X$ be a $\mathbb Q$-factorial $\epsilon$-lc Fano variety of dimension $n\geq 2$. If $q_{\mathbb Q}(X)\geq 3$, then 
    \[
         c_1(X)^n\leq \frac{2+2\epsilon}{\epsilon}\,\hat c_2(X)\cdot c_1(X)^{n-2}.
    \]
\end{rem}

\begin{rem}\label{rem.pseff}
    By \eqref{eq.genample} in the proof, we see that $\mu_{\max}(\mathcal{T}_X) = \mu(\mathcal{F})\leq c_1(X)^2\cdot D_1 \cdots D_{n-2}$ always holds, regardless of the Picard number. So \eqref{eq.keyineq} implies that for any nef divisors $D_1,\dots D_{n-2}$,
    \begin{equation}
        \hat{c}_2(X) \cdot D_1 \cdots D_{n-2}\geq 0
    \end{equation}
    and it is strict if $X$ is Fano and the $D_i$'s are ample. This inequality holds for klt singularities without restrictions on the $\mathbb Q$-Fano indices. It is known as the \emph{pseudoeffectivity of the generalized second Chern class} for klt weak Fano varieties (cf. \cite{ou}*{Corollary 1.5}).
\end{rem}

\section{Fano indices of Gorenstein canonical Fano  threefolds}\label{sec.Gorenstein}

In this section, we apply the Kawamata--Miyaoka type inequality to study the $\mathbb Q$-Fano indices of Gorenstein canonical Fano $3$-folds. Note that the problem cannot be reduced to the Picard number one case, as the Gorenstein property is not necessarily preserved when running the minimal model program.

\begin{thm}\label{thm.KMfor3fold}
    Let $X$ be a Gorenstein canonical Fano $3$-fold. Let $A$ be an ample Weil divisor such that $-K_X\sim qA$ for some positive integer $q$. Let $J$ be the smallest positive integer such that  $JA$ is Cartier in codimension $2$. Let $J=p_1^{a_1}p_2^{a_2}\cdots p_k^{a_k}$ be the prime factorization, where $p_i$ are distinct prime numbers. If $q\geq 3$, then
    \begin{align}
        \sum_{i=1}^k\left(p_i^{a_i}-\frac{1}{p_i^{a_i}}\right)\leq 24- \frac{c_1(X)^3}{4}.\label{eq upper J1}
    \end{align}
\end{thm}

\begin{proof}
    We use the notation $(e_C, g_C,j_C)$ in \cite{jiang-liu-liu}*{Definition 4.4}. Then by \cite{jiang-liu-liu}*{Theorem 4.6}, we have
    \begin{equation}\label{eq.diffofchernclass}
        \sum_{C\subset \Sing(X)}\left(e_C-\frac{1}{g_C} \right)(-K_X\cdot C)= c_2(X)\cdot c_1(X)-\hat c_2(X)\cdot c_1(X).
    \end{equation}
    As $X$ is Gorenstein, we have $c_2(X)\cdot c_1(X)=24$ by the Riemann--Roch formula and $-K_X\cdot C\geq 1$ for any $C\subset \Sing(X)$. Combining with \cite{jiang-liu-liu}*{Lemma 4.5}, we have
    \begin{equation}\label{eq.ineq1}
        \sum_{C\subset \Sing(X)}\left(j_C-\frac{1}{j_C} \right)\leq \sum_{C\subset \Sing(X)}\left(e_C-\frac{1}{g_C} \right)(-K_X\cdot C).
    \end{equation}
    As $J \mid \lcm \{j_C\mid C\subset \Sing(X)\}$ by the definitions of $J$ and $j_C$, each $p_i^{a_i}$ divides at least one $j_C$ for some $C\subset \Sing(X)$. Hence by \cite{chen-jiang}*{Page 65, (2.2)}, 
    \begin{equation}\label{eq.ineq2}
        \sum_{i=1}^k\left(p_i^{a_i}-\frac{1}{p_i^{a_i}}\right)\leq \sum_{C\subset \Sing(X)}\left(j_C-\frac{1}{j_C} \right).
    \end{equation}
    On the other hand, let $f\colon Y\to X$ be a $\mathbb Q$-factorialization such that $K_Y=f^*K_X$. In particular, $Y$ is a $\mathbb Q$-factorial Gorenstein canonical weak Fano $3$-fold such that $q_{\mathbb Q}(Y)=q_{\mathbb Q}(X)\geq q\geq 3$. As $X$ is $\mathbb Q$-factorial in codimension 2,  we have $f_*\hat c_2(Y)=\hat c_2(X)$ by pushing forward the one cycle, and hence $\hat c_2(Y)\cdot c_1(Y)=\hat c_2(X)\cdot c_1(X)$ by the projection formula. Then Corollary \ref{cor.KMinq} implies that
    \begin{equation}\label{eq.ineq3}
         c_1(X)^3=c_1(Y)^3\leq 4\hat c_2(Y)\cdot c_1(Y)=4\hat c_2(X)\cdot c_1(X).
    \end{equation}
    The conclusion follows by combining \eqref{eq.diffofchernclass}--\eqref{eq.ineq3}. 
\end{proof}

\begin{rem}\label{rem.surfaceindex}
    Let $S$ be a del Pezzo surface with du Val singularities and $q_{\mathbb Q}(S)\geq 3$. Then as above, Corollary \ref{cor.KMinq} and \cite{jiang-liu-liu}*{(4.5), (4.6)} yield that
    \begin{equation}\label{eq.KMsurface}
        \sum_{p\in \Sing(S)}(j_p-\frac{1}{j_p})\leq \sum_{p\in \Sing(S)}(e_p-\frac{1}{g_p})\leq c_2(S)-\frac{c_1(S)^2}{4}=12-\frac{5}{4}c_1(S)^2,
    \end{equation}
    where $e_p, g_p,j_p$ are defined as in \cite{jiang-liu-liu}*{Definition 4.4} and the last equality is Noether's formula. These strong numerical constraints provide a new approach to the classification of del Pezzo surfaces with du Val singularities. In particular, it yields a quick proof of $q_{\mathbb Q}(S)\leq 6$ that does not rely on the classification in \cite{miyanishi-zhang}, thereby offering an alternative to \cite{wang}*{Proposition 3.3}.

    Indeed, after taking quasi-\'etale cover if necessary, we can assume that $-K_S\sim qA$ for some ample Weil divisor $A$ and some integer $q\geq q_{\mathbb Q}(S)$. By \cite{jiang-liu-liu}*{Theorem 4.2}, we have $J\mid q $ and $q^2 \mid Jc_1(S)^2$. So $q\leq c_1(S)^2\leq 48/5<10$ by \eqref{eq.KMsurface}. Moreover, if $q>6$, then $6<c_1(S)^2=J=q\leq 9$. But $J\mid \lcm \{j_p\}$, which contradicts \eqref{eq.KMsurface} by a straightforward computation.
\end{rem}

\begin{rem}
    It follows from Remark \ref{rem.pseff} that $\hat c_2(X)\cdot c_1(X)> 0$ for any $\mathbb Q$-factorial Gorenstein canonical Fano $3$-fold. Combining with \eqref{eq.diffofchernclass} and \eqref{eq.ineq1}, we have
    \[
        \sum_{C\subset \Sing(X)}(j_C-\frac{1}{j_C})< 24.
    \]
    This inequality plays a role in the classification of Gorenstein canonical Fano $3$-folds analogous to that of the inequality $\sum_{r\in\mathcal R_X}(r-\frac{1}{r})<24$ in the terminal case. In particular, it allows us to enumerate all but finitely many candidates for the basket $\{j_C\}$.
\end{rem}

\begin{lem}\label{lem.reduction}
    Let $X$ be a Gorenstein canonical Fano $3$-fold such that $-K_X\sim_{\mathbb Q}qA$ for some integer $q$. Then there exists a Gorenstein canonical Fano $3$-fold $Y$ such that $\Cl(Y)$ is torsion free and $-K_Y\sim qA_Y$ for some Weil divisor $A_Y$ on $Y$.
\end{lem}

\begin{proof}
    Taking quasi-\'etale covers induced by torsion elements in $\Cl(X)$ repeatedly, we will obtain the desired $Y$, which is still a Gorenstein canonical Fano $3$-fold by \cite{kollar-mori}*{Proposition 5.20}. See the proof of \cite{jiang-liu-liu}*{Lemma 5.1} for more details.
\end{proof}

Theorem \ref{thm.KMfor3fold} imposes strong constraints on the numerical invariants of Gorenstein canonical Fano $3$-folds. The following result specifically restricts the possible Fano indices.

\begin{thm}\label{thm.listofindex}
    Let $X$ be a Gorenstein canonical Fano $3$-fold. 
    Let $A$ be an ample Weil divisor such that $-K_X\sim qA$ for some positive integer $q$. If $q>22$, then 
    \[
        q\in\{24,26,28,30,36,40,42\}.
    \]
\end{thm}

\begin{proof}
    Let $J$ be the smallest positive integer such that $JA$ is Cartier in codimension 2. Then $q^2\mid Jc_1(X)^3$ and $J\mid q$ by Theorem \ref{thm.degreeandindex}. In particular, $c_1(X)^3=q(\frac{q}{J})k$ for some positive integer $k$. By Theorem \ref{thm.degree}, we assume that $c_1(X)^3\leq 70$.
    
    If $q\geq 36$, then we must have $k=1$ and $J=q$. Hence $c_1(X)^3=J=q\leq 66$ is an even number by Theorem \ref{thm.degree} and \cite{jiang-liu2025}*{Theorem 1.1}. Then it follows from Theorem \ref{thm.KMfor3fold} and a straightforward computation that $q\in\{36,40,42\}$.

    If $22< q\leq 35$ and $q$ is odd, then $k=2$, $q=J$ and $c_1(X)^3=q(\frac{q}{J})k=2q$ as $c_1(X)^3\leq 70$ is even.  Such $q$ is excluded by Theorem \ref{thm.KMfor3fold}. 

    For $q=34$ and $32$, we have $c_1(X)^3=q=J$, $c_1(X)^3=2q=2J$, or $c_1(X)^3=2q=4J$. All of these possibilities are excluded by Theorem \ref{thm.KMfor3fold} as well.
\end{proof}

\begin{rem}\label{rem.RR}
    The cases $q=19$ and $17$ can be ruled out by repeating the previous computations. Alternatively, some of the cases can also be excluded using the Riemann--Roch formula for the Gorenstein canonical Fano $3$-folds \cite{jiang-liu2025}*{Theorem 4.9}, which states that for any integer $0<s<q$, 
    \begin{equation}\label{eq.RR}
        \frac{s^2}{2q^2}c_1(X)^3+\sum_{C\subset \Sing(X)}(-K_X\cdot C)c_C(sA)=h^0(X,\mathcal O_X(sA))-2\in\mathbb Z.
    \end{equation}
    With Theorem \ref{thm.KMfor3fold} and \eqref{eq.RR} at hand, we can rule out many cases by following the same line of argument as in \cite{jiang-liu2025}*{\S~7}. For instance, in the case $q=40$, we have $c_1(X)^3=q=J$, and hence Theorem \ref{thm.KMfor3fold} 
    implies that the only singular curves are $C_1$ of type $A_4$ and $C_2$ of type $A_7$, both satisfying $-K_X\cdot C_i=1$. Substituting into \eqref{eq.RR} yields that for any $0<s<40$, 
    \[
        \frac{s^2}{80}-\frac{\overline{sx}(5-\overline{sx})}{10}-\frac{\overline{sy}(8-\overline{sy})}{16}\in \mathbb Z
    \]
    for some integers $0\leq x<5, 0\leq y<8$. However, this fails for $s=5$. A similar argument applies to $q=36, 28, 26$.
\end{rem}

\begin{ex}\label{ex.wps}
    According to \cite{ds}*{Table 1}, there are 14 Gorenstein $3$-dimensional weighted projective spaces, whose $\mathbb Q$-Fano indices form the set
    \[
        \{4, 6, 8, 10, 12, 18,20,24,30,42\}.
    \]
    The case $q_{\mathbb Q}(X)=24$ is realized by $X\cong \mathbb P^1(1,3,8,12)$, the case $q_{\mathbb Q}(X)=30$ is realized by $X\cong \mathbb P^1(2,3,10,15)$, and the case $q_{\mathbb Q}(X)=42$ is realized by $X\cong \mathbb P^1(1,6,14,21)$. 
\end{ex}

\begin{cor}\label{cor.index}
    Let $X$ be a Gorenstein canonical Fano $3$-fold. Then 
    \[
        q_{\mathbb Q}(X)\in  \{m\in\mathbb Z_{>0}\mid m\leq 22\} \cup\{24,30,42\}\subset \{m\in\mathbb Z_{>0}\mid \varphi(m)\leq 20\} \setminus \{60\}.
    \]
\end{cor}

\begin{proof}
    By Lemma \ref{lem.reduction}, there exists a Gorenstein canonical Fano $3$-fold $Y$ such that $-K_Y\sim q_{\mathbb Q}(X)A$ for some Weil divisor $A$ on $Y$. We may assume that $q_{\mathbb Q}(X)>22$. The conclusion then follows directly from Theorem \ref{thm.listofindex} and Remark \ref{rem.RR}.
\end{proof}

\section*{Acknowledgments} 
The author would like to thank Chen Jiang, Jie Liu and Wenhao Ou for helpful discussions and comments. He is especially grateful to Jie Liu for kindly permitting the use of his result in Proposition \ref{prop.smoothq=2}. The author is supported by the National Key Research and Development Program of China (No. 2023YFA1009801) in part and by NSFC (No. 12571048).  



\begin{thebibliography}{99} 

\bibitem{aroujo-druel}
C.~Araujo, S.~Druel,
\emph{On Fano foliations},
Adv. Math. \textbf{238} (2013), 70--118.

\bibitem{cp} 
F.~Campana, M.~P{\u a}un, 
\emph{Foliations with positive slopes and birational stability of orbifold cotangent bundles}, 
Publ. Math. Inst. Hautes \'Etudes Sci.  \textbf{129} (2019), 1--49.

\bibitem{chen-jiang} 
M.~Chen, C.~Jiang, 
\emph{On the anti-canonical geometry of $\mathbb Q$-Fano threefolds}, 
J. Differential Geom. \textbf{104} (2016), no. 1, 59--109. 

\bibitem{ds} 
T.~Dedieu, E.~Sernesi, 
\emph{Deformations and extensions of Gorenstein weighted projective spaces}, in 
\emph{The art of doing algebraic geometry}, 
Trends in Mathematics, Birkh\"auser/Springer, Cham, 2023, 119--143.

\bibitem{druel2017}
S.~Druel,
\emph{On foliations with nef anti-canonical bundle},
Trans. Amer. Math. Soc., \textbf{369} (2017), no. 11, 7765--7787.

\bibitem{druel2021} 
T.~Druel, 
\emph{Codimension $1$ foliations with numerically trivial canonical class on singular spaces}, 
Duke Math. J. \textbf{170} (2021), no. 1, 95--203.

\bibitem{fujino}
O.~Fujino,
\emph{The indices of log canonical singularities},
Amer. J. Math. \textbf{123} (2001), no. 2, 229--253.

\bibitem{gkp}
D.~Greb, S.~Kebekus, Th.~Peternell,
\emph{Projectively flat klt varieties},
J. \'Ec. polytech. Math. \textbf{8} (2021), 1005--1036.

\bibitem{gkpt}
D.~Greb, S.~Kebekus, Th.~Peternell, B.~Taji,
\emph{The Miyaoka--Yau inequality and uniformisation of canonical models},
Ann. Sci. \'Ec. Norm. Sup\'er. (4) \textbf{52} (2019), no. 6, 1487--1535.



\bibitem{ip} 
V.~Iskovskikh, Y.~Prokhorov,
\emph{Algebraic geometry. V},
volume 47 of \emph{Encyclopaedia of Mathematical Sciences}. Springer-Verlag, Berlin, 1999.
Fano varieties, A translation of ıt Algebraic geometry. 5 (Russian), Ross. Akad. Nauk,
Vseross. Inst. Nauchn. i Tekhn. Inform., Moscow, Translation edited by A.~N.~Parshin
and I.~R.~Shafarevich.

\bibitem{ijl} 
M.~Iwai, C.~Jiang, H.~Liu,
 \emph{Miyaoka type inequality for terminal threefolds with nef anti-canonical divisors},
 Sci. China Math. \textbf{68} (2025), no. 1, 1--18.

\bibitem{jiang-liu2025}
C.~Jiang, H.~Liu, 
\emph{A canonical Fano threefold has Fano index $\leq 66$}, 
arXiv:2508.16364v2.

\bibitem{jiang-liu-liu}
C.~Jiang, H.~Liu, J.~Liu,
\emph{Optimal upper bound for degrees of canonical Fano threefolds of Picard number one}, 
arXiv:2501.16632v2.

\bibitem{kasprzyk}
A.~M.~Kasprzyk. 
\emph{Canonical toric Fano threefolds},
Canad. J. Math. \textbf{62} (2010), no. 6, 1293--1309.

\bibitem{keel-matsuki-mckernan}
S.~Keel, K.~Matsuki, J.~M\textsuperscript{c}Kernan, 
{\it Log abundance theorem for threefolds},
Duke Math. J. {\bf 75} (1994), no. 1, 99--119.

\bibitem{kollar}
J.~Koll\'{a}r, 
\emph{Rational curves on algebraic varieties},
Results in Mathematics and Related Areas. 3rd Series. A series of Modern Surveys in Mathematics, vol. 32, Springer-Verlag, Berlin, 1996.

\bibitem{kollar-mori}
J.~Koll\'{a}r, S.~Mori, 
\emph{Birational geometry of algebraic varieties},
Cambridge tracts in mathematics, vol. 134, Cambridge University
Press, 1998.

\bibitem{lazarsfeld}
R.~Lazarsfeld,
\emph{Positivity in algebraic geometry. II. Positivity for vector bundles, and multiplier ideals},
Results in Mathematics and Related Areas. 3rd Series. A series of Modern Surveys in Mathematics, vol. 49, Springer-Verlag, Berlin, 2004.

\bibitem{liu-liu2023}
H.~Liu, J.~Liu, 
\emph{Kawamata--Miyaoka type inequality for $\mathbb Q$-Fano varieties with canonical singularities},
J. Reine Angew. Math. \textbf{819} (2025), 265--281.

\bibitem{liu-liu2024}
H.~Liu, J.~Liu, 
\emph{Kawamata--Miyaoka-type inequality for $\mathbb Q$-Fano varieties with canonical singularities II: Terminal $\mathbb Q$-Fano threefolds},
\'Epijournal G\'eom. Alg\'ebrique \textbf{9} (2025), Art. 12, 21pp.

\bibitem{machida-oguiso}
N.~Machida, K.~Oguiso, 
\emph{On K3 surfaces admitting finite non-symplectic group actions}, 
J. Math. Sci. Univ. Tokyo \textbf{5} (1998), no. 2, 273--297.

\bibitem{masamura}
Y.~Masamura, 
\emph{Relations between indices of Calabi--Yau varities and pairs}, 
Int. Math. Res. Not. IMRN (2025), no. 10,1--16.

\bibitem{miyanishi-zhang}
M.~Miyanishi, D.-Q.~Zhang. 
\emph{Gorenstein log del Pezzo surfaces of rank one},
J. Algebra \textbf{118} (1988), no. 1, 63--84.
 
\bibitem{ou}
W.~Ou,
\emph{On generic nefness of tangent sheaves}, 
Math. Z. \textbf{304} (2023), no. 4, Paper No. 58.

\bibitem{prokhorov2005} 
Y.~Prokhorov, 
\emph{The degree of Fano threefolds with canonical Gorenstein singularities}, 
Mat. Sb. \textbf{196} (2005), no. 1, 81--122; translation in Sb. Math. \textbf{196} (2005), no. 1--2, 77--114.

\bibitem{prokhorov2010} 
Y.~Prokhorov, 
\emph{$\mathbb Q$-Fano threefolds of large Fano index, I}, 
Doc. Math. \textbf{15} (2010), 843--872.

\bibitem{wang}
C.~Wang,
\emph{Fano varieties with conjecturally largest Fano index}, 
Internat. J. Math. \textbf{35} (2024), no. 12, Paper No. 2450048, 16 pp.

\bibitem{ye}
Q.~Ye,
\emph{On Gorenstein log del Pezzo surface},
Japan. J. Math. (N.S.) \textbf{28} (2002), no. 1, 87--136.

\end{thebibliography}
\end{document}